\documentclass[a4paper]{amsart}
\usepackage{amssymb,amsmath,bbm}
\usepackage{hyperref}

\newtheorem{theorem}{Theorem}[section]
\newtheorem{corollary}[theorem]{Corollary}
\newtheorem{lemma}[theorem]{Lemma}
\newtheorem{proposition}[theorem]{Proposition}

\theoremstyle{definition}
\newtheorem{example}{Example}[section]
\newtheorem{remark}{Remark}[section]
\begin{document}
\sloppy

\author{Janusz Morawiec}
\email[J. Morawiec]{janusz.morawiec@us.edu.pl}
\title{An application of medial limits to iterative functional equations}
\address{Institute of Mathematics, University of Silesia, Bankowa 14, 40-007 Katowice, Poland}
\keywords{Banach limits, medial limits, iterative functional equations, bounded solutions}

\begin{abstract}
Assume that $(\Omega,\mathcal A,P)$ is a probability space, $f\colon[0,1]\times\Omega\to[0,1]$ is a function such that $f(0,\omega)=0$, $f(1,\omega)=1$ for every $\omega\in\Omega$, $g\colon[0,1]\to\mathbb R$ is a bounded function such that $g(0)=g(1)=0$, and $a,b\in\mathbb R$. Applying medial limits we describe bounded solutions $\varphi\colon[0,1]\to\mathbb R$ of the equation
\begin{equation*}
\varphi(x)=\int_{\Omega}\varphi(f(x,\omega))dP(\omega)+g(x)
\end{equation*} 
satisfying the boundary conditions $\varphi(0)=a$ and $\varphi(1)=b$.
\end{abstract}

\maketitle


\section{Introduction}
Assume that $(\Omega,\mathcal A,P)$ is a probability space, $f\colon[0,1]\times\Omega\to[0,1]$ is a function satisfying the boundary condition
\begin{equation}\label{cond}
f(0,\omega)=0\quad\hbox{and}\quad f(1,\omega)=1\quad\hbox{for every }\omega\in\Omega, 
\end{equation} 
$g\colon[0,1]\to\mathbb R$ is a bounded function with $g(0)=g(1)=0$ and $a,b\in\mathbb R$ are fixed numbers. We are interested in bounded solutions $\varphi\colon[0,1]\to\mathbb R$ of the following iterative functional equation
\begin{equation}\label{E}
\varphi(x)=\int_{\Omega}\varphi(f(x,\omega))dP(\omega)+g(x)\tag{E$_g$}.
\end{equation}
We say that a function $\varphi\colon[0,1]\to\mathbb R$ is a solution of equation \eqref{E} if for every $x\in[0,1]$ the function $\varphi\circ f(x,\cdot)$ is measurable and \eqref{E} holds.

The main purpose of this paper is to describe all solutions of equation \eqref{E} in some classes of bounded functions $h\colon[0,1]\to\mathbb R$ such that $h(0)=a$ and $h(1)=b$. We are also interested under which assumptions any bounded solution $\varphi\colon[0,1]\to\mathbb R$ of equation \eqref{E} with a certain property can be expressed in the form $\varphi=\Phi+\varphi_*$, where $\Phi$ is a solution of the equation
\begin{equation}\label{E0}
\Phi(x)=\int_{\Omega}\Phi(f(x,\omega))dP(\omega)\tag{E$_0$}
\end{equation}
having the same property as $\varphi$ and $\varphi_*$ is a specific solution of equation \eqref{E}. This problem seems to be easy to answer, but the difficulty is that the classes considered in this paper are not linear spaces. It is even not clear when the existence of a solution with a certain property of one of the equations \eqref{E} and \eqref{E0} implies the existence of a solution with the same property of the other of these equations. Such a problem is quite natural in the theory of functional equations and it has been studied several times by many authors for different functional equations in various classes of functions; mainly in cases where the class of considered functions forms a vector space. 

Functional equations \eqref{E0} and \eqref{E}, as well as their generalizations and special cases, are investigated in various classes of functions in connection with their appearance in miscellaneous fields of science (for more details see \cite[Chapter XIII]{K1968}, \cite[Chapters 6, 7]{KCG1990} and \cite[Section 4]{BJ2001}). 
As emphasized in \cite[section 0.3]{KCG1990} iteration is the fundamental technique for solving functional equations in a single variable, and iterates usually appear in the formulae for solutions. In most cases such formulas are obtained by taking the limit of sequences in which iterates are involved. In this paper we make use of this fundamental technique, but the goal is to apply a subclass of Banach limits instead of the limit. The idea of replacing the limit by a Banach limit seems to be clear, because we do not need any additional assumption guaranteeing the existence of a Banach limit of a bounded sequence, in contrast to the case when we want to calculate the limit of such a sequence. 

This paper is organized as follows. Section 2 contains the notation and basic tools required for our considerations. In sections 3 and 4 we describe bounded solutions $\varphi\colon[0,1]\to\mathbb R$ with $\varphi(0)=a$ and $\varphi(1)=b$ of equations \eqref{E0} and \eqref{E}, respectively. Finally, in section 5, we formulate some consequences of the main results obtained and we present a few examples of the possible applications of those results.


\section{Preliminaries}
Denote by $B([0,1],\mathbb R)$ the space of all bounded functions $h\colon[0,1]\to\mathbb R$ endowed with the supremum norm and respectively by $BM([0,1],\mathbb R)$, $C_x([0,1],\mathbb R)$, $Lip([0,1],\mathbb R)$ and $BV([0,1],\mathbb R)$ its subspaces of all functions that are Borel measurable, continuous at $x\in[0,1]$, Lipschitzian, and of bounded variation (i.e. functions which can be written as a difference of two increasing functions; see \cite[Chapter 1.4]{L1988}). Next denote by $\mathcal M([0,1],\mathbb R)$ the space consisting of all functions $h\in B([0,1],\mathbb R)$ such that for every $x\in[0,1]$ the function $h\circ f(x,\cdot)$ is measurable. Note that the space $\mathcal M([0,1],\mathbb R)$ is at most one dimensional, because every constant function belongs to it.

Define an operator $T\colon\mathcal M([0,1],\mathbb R)\to B([0,1],\mathbb R)$ by setting
\begin{equation*}\label{Th}
Th(x)=\int_{\Omega}h(f(x,\omega))dP(\omega).  
\end{equation*}
Note that $T$ is linear and continuous with $\|T\|=1$. Moreover, equation \eqref{E} can be written now in the form
\begin{equation}\label{eT}
\varphi=T\varphi+g.
\end{equation}
To the end of this paper we fix a subspace $\mathcal B([0,1],\mathbb R)$ of the space $\mathcal M([0,1],\mathbb R)$ that is invariant under $T$, i.e. $\mathcal B([0,1],\mathbb R)$ is such that
\begin{equation}\label{T}
T(\mathcal B([0,1],\mathbb R))\subset \mathcal B([0,1],\mathbb R).
\end{equation}

Before we give examples showing how the space $\mathcal B([0,1],\mathbb R)$ can look like in certain situations, let us say what we mean writing ${\mathcal A}=2^{\Omega}$. Namely, in such a case we may (and do) assume that $\Omega$ is countable; otherwise we can replace $\Omega$ by its subset $\{\omega\in\Omega: P(\{\omega\})>0\}$, which is clearly countable. Moreover, integration in \eqref{E} reduces to summation and every bounded function is measurable.	

\begin{example}\label{ex1}
If ${\mathcal A}=2^{\Omega}$, then \eqref{T} holds with $\mathcal B([0,1],\mathbb R)=B([0,1],\mathbb R)$.	
\end{example}

\begin{example}\label{ex5}
If
\begin{enumerate}
\item[(H$_1$)] $f$ is increasing with respect to the first variable and measurable with respect to the second variable,
\end{enumerate} 
then \eqref{T} holds with $\mathcal B([0,1],\mathbb R)=BV([0,1],\mathbb R)$. 
\end{example}

\begin{example}\label{ex4}	
If
\begin{enumerate}
\item[(H$_2$)] $\int_{\Omega}|f(x,\omega)-f(y,\omega)|dP(\omega)\leq|x-y|$ for all $x,y\in [0,1]$ and $f$ is measurable with respect to the second variable,
\end{enumerate} 
then \eqref{T} holds with $\mathcal B([0,1],\mathbb R)=Lip([0,1],\mathbb R)$.
\end{example}

Let $\mathcal B$ be the $\sigma$-algebra of all Bores subsets of $[0,1]$. Following \cite{BK1977} we say that $h\colon[0,1]\times\Omega\to[0,1]$ is a
random-valued function (shortly: an rv-function) if it is measurable with respect to $\sigma$-algebra $\mathcal B\times\mathcal A$.

\begin{example}\label{ex7}
If $f$ is an rv-function, then \eqref{T} holds with $\mathcal B([0,1],\mathbb R)=BM([0,1],\mathbb R)$. 
\end{example}

\begin{example}\label{ex2}
Fix $x_0\in\{0,1\}$ and let $f(\cdot,\omega)$ be continuous at $x_0$ for every $\omega\in\Omega$. If ${\mathcal A}=2^{\Omega}$, then \eqref{T} holds with $\mathcal B([0,1],\mathbb R)=C_{x_0}([0,1],\mathbb R)$. If $f$ is an rv-function, then \eqref{T} holds with $\mathcal B([0,1],\mathbb R)=BM([0,1],\mathbb R)\cap C_{x_0}([0,1],\mathbb R)$.
\end{example}

To describe solutions of equation \eqref{E} in the case where $\mathcal A=2^\Omega$ we need the concept of Banach limits, established in \cite{B1955}. However in the general case, when integration is required, we need the concept of medial limits, established in \cite{M1969} (cf. \cite{M1973}) as a very special class of Banach limits. 

Denote by $l^\infty(\mathbb N)$ the space of all bounded real sequences equipped with the supremum norm and by $\mathfrak B$ the family of all Banach limits defined on $l^\infty(\mathbb N)$. Recall that $B\in\mathfrak B$ if $B\colon l^\infty(\mathbb N)\to\mathbb R$ is a linear, positive, shift invariant and normalized operator. It is easy to see that any $B\in\mathfrak B$ is continuous with $\|B\|=1$. It is known that the cardinality of $\mathfrak B$ is equal to $2^{\mathfrak c}$ (see \cite{D1965}), and even that the cardinality of the set of all extreme points of $\mathfrak B$ is equal to $2^{\mathfrak c}$ (see \cite{C1969}, cf. \cite{P2014}); here $\mathfrak c$ is the cardinality of the continuum.
 
As it was mentioned above, in the general case we need to integrate the pointwise Banach limit of a bounded sequence of measurable functions. However, the problem is that there is no guarantee that the pointwise Banach limit of a bounded sequence of measurable functions is a measurable function (see \cite[page 288]{W2000}). Fortunately, it is known that there are Banach limits possessing exactly the required property. More precisely, a Banach limit $B$ is called a medial limit if $\int_{\Omega}B((h_m(\omega))_{m\in\mathbb N})dP(\omega)$ is defined and equal to $B((\int_{\Omega}h_m(\omega)dP(\omega))_{m\in\mathbb N})$ whenever $(h_m)_{m\in\mathbb N}$ is a bounded sequence of measurable real-valued functions on $\Omega$. It is also known that the continuum hypothesis implies the existence of medial limits. More results on the existence and non-existence of medial limits can be found in \cite[Chapter 53]{F2015} and in \cite{L2009}. 
Denote by $\mathfrak M$ the family of all medial limits, i.e. $B\in\mathfrak M\subset\mathfrak B$ if
\begin{equation*}
\int_{\Omega}B\big((h_m(\omega))_{m\in\mathbb N}\big)dP(\omega)
=B\left(\Big(\int_{\Omega}h_m(\omega)dP(\omega)\Big)_{m\in\mathbb N}\right)
\end{equation*}
for every sequence $(h_m)_{m\in\mathbb N}$ of bounded measurable real-valued functions defined on $\Omega$. Note that $\mathfrak M=\mathfrak B$ in the case where $\mathcal A=2^\Omega$.

From now on, given a nonempty family $\mathcal F\subset B([0,1],\mathbb R)$ we denote by $\mathcal F_a^b$ the family of all $h\in\mathcal F$ such that $h(0)=a$ and $h(1)=b$. To distinguish two important families let us adopt the shorthand $\mathcal M_a^b=\mathcal M([0,1],\mathbb R)_a^b$ and $\mathcal B_a^b=\mathcal B([0,1],\mathbb R)_a^b$. It is clear that $\mathcal B_0^0$ is a subspace of the space $\mathcal B([0,1],\mathbb R)$ and that $\mathcal B_a^b+\mathcal B_0^0=\mathcal B_a^b$. It is also clear that if we determine all solutions of equation \eqref{E} in the class $\mathcal B_a^b$, then we can easily describe all solutions of this equation in the class $\mathcal B([0,1],\mathbb R)$. Now we are in a position to begin describing solutions of equation \eqref{E} in the class $\mathcal B_a^b$. Our first lemma is a simple consequence of \eqref{cond} and \eqref{T}.

\begin{lemma}\label{lempropTB}
\begin{enumerate}
\item[\rm (i)] If $h\in\mathcal B_a^b$, then $Th\in\mathcal B_a^b$. 
\item[\rm (ii)] If $\varphi\in\mathcal B_a^b$ satisfies \eqref{E}, then $g\in\mathcal B_0^0$.
\end{enumerate}
\end{lemma}


\section{Solutions of equation \eqref{E0}}
If $h\in\mathcal B([0,1],\mathbb R)$, then $\sup_{m\in\mathbb N}\|T^mh\|\leq\|h\|$, and hence
$(T^mh(x))_{m\in\mathbb N}\in l^\infty(\mathbb N)$ for every $x\in [0,1]$. Therefore, for all $h\in\mathcal B([0,1],\mathbb R)$ and $B\in\mathfrak B$ we define a function $B_h\colon[0,1]\to\mathbb R$ by putting
\begin{equation*}
B_h(x)=B\big((T^mh(x))_{m\in\mathbb N}\big).
\end{equation*}
The functions $B_h$ plays a crucial role in this section as well as in this paper. So, we need some fact about them. 

\begin{lemma}\label{lemBh}
Assume that $h\in\mathcal B_a^b$. If $B\in\mathfrak M$, then $B_h\in \mathcal M_a^b$ and $TB_h=B_h$.
\end{lemma}

\begin{proof}
Fix $B\in\mathfrak M$. From Lemma \ref{lempropTB}(i) we see that $T^mh\in \mathcal B_a^b$ for every $m\in\mathbb N$. Then 
\begin{equation*}
\sup_{x\in[0,1]}|B_h(x)|\leq\sup_{x\in[0,1]}\|B\|\sup_{m\in\mathbb N}|T^mh(x)|\leq\|h\|.
\end{equation*}
Thus $B_h\in B([0,1],\mathbb R)$. Since $B\in\mathfrak M$, it follows that $B_h\in\mathcal M([0,1],\mathbb R)$. Moreover, \eqref{cond} implies $B_h(0)=a$ and $B_h(1)=b$. In consequence, $B_h\in \mathcal M_a^b$. 

Applying properties of medial limits we obtain
\begin{align*}
TB_h(x)&=\int_{\Omega}B\left(\big(T^mh(f(x,\omega))\big)_{m\in\mathbb N}\right)dP(\omega)\\
&=B\left(\Big(\int_{\Omega}T^mh(f(x,\omega))dP(\omega)\Big)_{m\in\mathbb N}\right)
=B\left(\big(T^{m+1}h(x)\big)_{m\in\mathbb N}\right)=B_h(x)
\end{align*}
for every $x\in[0,1]$.
\end{proof}

We now want to find conditions under which $B_h\in\mathcal B_a^b$ for every $h\in\mathcal B_a^b$. Unfortunately, there is no chance to prove that $B_h\in\mathcal B_a^b$ in the general case. In fact, we would have to show that \eqref{T} holds, i.e. $TB_h\in\mathcal B_a^b$, but by Lemma \ref{lemBh} we have $TB_h=B_h$. This observation suggests the following definition. 

We say that the class $\mathcal B_a^b$ is {\it closed under $B\in\mathfrak B$}, if $B_h\in\mathcal B_a^b$ for every $h\in\mathcal  B_a^b$. 

It turns out that there are many interesting classes that are closed under some Banach limits. Let us now give a few examples of such classes. The first two are immediate consequences of Examples \ref{ex1} and \ref{ex7}.

\begin{example}\label{ex1a}
If $\mathcal B([0,1],\mathbb R)=B([0,1],\mathbb R)$, then $B([0,1],\mathbb R)_a^b$ is closed under any $B\in\mathfrak B$.
\end{example}

\begin{example}\label{ex7a}
If $f$ is an rv-function and $B$ is a medial limit with respect to a probability Borel measure on $[0,1]$, then $BM([0,1],\mathbb R)_a^b$ is closed under $B$. 
\end{example}

\begin{example}\label{ex2a}
Fix $x_0\in\{0,1\}$. If $\mathcal A=2^\Omega$ and
\begin{enumerate}
\item[(H$_{3}$)] there exists $\eta>0$ such that $\frac{f(x,\omega)-f(x_0,\omega)}{x-x_0}\leq 1$ for all $\omega\in\Omega$ and $x\in(0,1)$ with $|x-x_0|\leq\eta$,
\end{enumerate} 
then $C_{x_0}([0,1],\mathbb R)_a^b$ is closed under any $B\in\mathfrak B$.
To prove the conclusion let us put $\mathcal B([0,1],\mathbb R)=C_{x_0}([0,1],\mathbb R)$; this is possibly according to Example \ref{ex2}, because (H$_3$) yields the continuity of $f(\cdot,\omega)$ at $x_0$ for every $\omega\in\Omega$. Let $B\in\mathfrak B$ and let  $h\in C_{x_0}([0,1],\mathbb R)_a^b$. It is clear that $B_h(0)=a$ and $B_h(1)=b$. To prove that $B_h$ is continuous at $x_0$ fix $\varepsilon>0$. Then choose $\delta\in(0,\eta)$, where $\eta$ is a number occurring in (H$_4$), such that 
\begin{equation}\label{hxhx0}
\sup_{x\in A}|h(x)-h(x_0)|\leq\varepsilon,
\end{equation} 
where $A=\{x\in[0,1]:|x-x_0|\leq\delta\}$. Then $|f(x,\omega)-f(x_0,\omega)|\leq\delta$ for all $\omega\in\Omega$ and $x\in A$, and by an easy induction, we obtain $\sup_{x\in A}|T^mh(x)-T^mh(x_0)|\leq\sup_{x\in A}|h(x)-h(x_0)|$ for every $m\in\mathbb N$. This jointly with \eqref{hxhx0} gives
\begin{align*}
\sup_{x\in A}|B_h(x)-B_h(x_0)|&\leq\sup_{x\in A}\|B\|\sup_{m\in\mathbb N}|T^mh(x)-T^mh(x_0)|\leq\varepsilon,
\end{align*}
which proves that $B_h$ is continuous at $x_0$.
\end{example}

\begin{example}\label{ex4a}
If (H$_2$) holds, then $Lip([0,1],\mathbb R)_a^b$ is closed under any $B\in\mathfrak B$. For the prove of the conclusion we put $\mathcal B([0,1],\mathbb R)=Lip([0,1],\mathbb R)$; this is acceptable according to Example \ref{ex4}. Fix $B\in\mathfrak B$ and $h\in Lip([0,1],\mathbb R)_a^b$. Clearly, $B_h(0)=a$ and $B_h(1)=b$. To prove that $B_h$ is Lipschitzian denote by $L$ the Lipschitz constant of $h$. A simple induction gives $|T^mh(x)-T^mh(y)|\leq L|x-y|$ for all $m\in\mathbb N$ and $x,y\in[0,1]$. Thus,
\begin{align*}
|B_h(x)-B_h(y)|&\leq\|B\|\sup_{m\in\mathbb N}|T^mh(x)-T^mh(y)|\leq L|x-y|
\end{align*}
for all $x,y\in[0,1]$.
\end{example}

\begin{example}\label{ex5a}
If (H$_1$) holds, then $BV([0,1],\mathbb R)_a^b$ is closed under any $B\in\mathfrak B$. To show that the conclusion holds we put $\mathcal B([0,1],\mathbb R)=BV([0,1],\mathbb R)$; this is possible according to Example \ref{ex5}. Fix $B\in\mathfrak B$ and $h\in BV([0,1],\mathbb R)_a^b$. Obviously, $B_h(0)=a$ and $B_h(1)=b$. Moreover, there exist increasing functions $h_1,h_2\in B([0,1],\mathbb R)$ such that $h=h_1-h_2$. Thus $T^mh=T^mh_1-T^mh_2$ for every $m\in\mathbb R$, and hence $B_h=B_{h_1}-B_{h_2}$. Finally, (H$_1$) jointly with properties of Banach limits implies that both the functions $B_{h_1}$ and $B_{h_2}$ are increasing.
\end{example}

We are now in a position to formulate the main results of this section. To simplify their statements, let us denote by $\mathfrak{sol}_a^b\eqref{E0}$ the family of all functions from $\mathcal B_a^b$ satisfying equation \eqref{E0}.

\begin{theorem}\label{thmmainE0}
For every $B\in\mathfrak M$ we have 
$\mathfrak{sol}_a^b\eqref{E0}\subset\left\{B_h:h\in\mathcal B_a^b\right\}$.
Moreover, if $\mathcal B_a^b$ is closed under $B\in\mathfrak M$, then 
$\mathfrak{sol}_a^b\eqref{E0}=\left\{B_h:h\in\mathcal B_a^b\right\}$.
\end{theorem}

\begin{proof}
Fix $B\in\mathfrak M$ and $\Phi\in\mathfrak{sol}_a^b\eqref{E0}$. Then $T^m\Phi=\Phi$ for every $m\in\mathbb N$, and hence
\begin{equation*}
\Phi(x)=B\big((T^m\Phi(x))_{m\in\mathbb N}\big)=B_\Phi(x)
\end{equation*}
for every $x\in[0,1]$. Thus $\mathfrak{sol}_a^b\eqref{E0}\subset\left\{B_h:h\in\mathcal B_a^b\right\}$. The opposite inclusion follows from Lemma \ref{lemBh}. 
\end{proof}

Solutions of equations \eqref{E0} was investigated in \cite{M1995,M1998}, basically in almost the same classes of bounded functions. However, Theorem \ref{thmmainE0} is incomparable with the results obtained in the papers mentioned, in which the existence and the uniqueness problems have been considered as well as properties of the unique solution have been studied.


\section{Solutions of equation \eqref{E}}
In this section we describe all functions belonging to the class $\mathcal B_a^b$ which are solutions of equation \eqref{E}. We also give the formula for these solutions showing that each of them can be written in the form $\Phi+\varphi_*$, where $\Phi\in\mathcal B_a^b$ is a solution of equation \eqref{E0} and  $\varphi_*\in\mathcal B_0^0$ is a particular solution of equation \eqref{E}. To find $\varphi_*$ we need define a certain family of functions generated by $g\in\mathcal B_0^0$; recall that $g\in\mathcal B_0^0$ is a necessary condition for equation \eqref{E} to have a solution in the class $\mathcal B_a^b$ by Lemma \ref{lempropTB}(ii). If $g\in\mathcal B_0^0$, then Lemma \ref{lempropTB}(i) yields $\{T^lg:l\in\mathbb N\}\subset\mathcal B_0^0$. Therefore, given $g\in\mathcal B_0^0$ and $k\in\mathbb N$ we can define a function $g_k\colon[0,1]\to\mathbb R$ by putting
\begin{equation*}
g_k(x)=\sum_{l=0}^{k-1}T^lg(x).
\end{equation*} 
Set
\begin{equation*}
\mathcal G=\{g_k:k\in\mathbb N\}.
\end{equation*}

As in the previous section, denote by $\mathfrak{sol}_a^b\eqref{E}$ the family of all functions from $\mathcal B_a^b$ satisfying equation \eqref{E}.

\begin{lemma}\label{lemG}
If $\mathfrak{sol}_a^b\eqref{E}\neq\emptyset$, then $\mathcal G$ is a bounded subset of $\mathcal B_0^0$.
\end{lemma}

\begin{proof}
Fix $\varphi\in\mathfrak{sol}_a^b\eqref{E}$. Then Lemma \ref{lempropTB} implies that $\mathcal G\subset\mathcal B_0^0$. Applying \eqref{eT} we obtain
\begin{align*}
\|g_k\|&=\sup_{x\in[0,1]}|g_k(x)|=\sup_{x\in[0,1]}\left|\sum_{l=0}^{k-1}T^{l}\varphi(x)-\sum_{l=0}^{k-1}T^{l+1}\varphi(x)\right|\\
&\leq\sup_{x\in[0,1]}|\varphi(x)|+\sup_{x\in[0,1]}|T^{k}\varphi(x)|\leq \|\varphi\|+\|T\|^{k}\|\varphi\|= 2\|\varphi\|
\end{align*}
for every $k\in\mathbb N$.
\end{proof}

The above lemma shows that boundedness of the family $\mathcal G$ is a necessary condition for equation \eqref{E} to have a solution in the class $\mathcal B_a^b$. This also demonstrate, that $B_{g_k}$ is well defined for all $k\in\mathbb N$ and $B\in\mathfrak B$ whenever equation \eqref{E} has a solution in $\mathcal B_a^b$. 

\begin{lemma}\label{lemBg}
If $\mathfrak{sol}_a^b\eqref{E}\neq\emptyset$, then $B_{g_k}=0$ for all $B\in\mathfrak B$ and $k\in\mathbb N$. 
\end{lemma}

\begin{proof}
Fix $\varphi\in\mathfrak{sol}_a^b\eqref{E}$, $B\in\mathfrak B$ and $k\in\mathbb N$. By \eqref{eT} we get
\begin{equation*}
B_{g}(x)=B((T^{m}g(x))_{m\in\mathbb N})=B((T^{m}\varphi(x))_{m\in\mathbb N})-B((T^{m+1}\varphi(x))_{m\in\mathbb N})=0
\end{equation*}
for every $x\in[0,1]$. Now, it only remains to see that $B_{g_k}=kB_{g}$.
\end{proof}

If $g\in\mathcal B_0^0$ and $\mathcal G$ is bounded, then for every $B\in\mathcal B$ we define a function $B_*\colon[0,1]\to\mathbb R$ by putting
\begin{equation*}\label{B*}
B_*(x)=B((g_k(x))_{k\in\mathbb N}).
\end{equation*}

\begin{lemma}\label{lemB*}
Assume that $g\in\mathcal B_0^0$ and $\mathcal G$ is bounded. If $B\in\mathfrak M$, then $B_*\in\mathcal M_0^0$ and $B_*=TB_*+g$.
\end{lemma}

\begin{proof}
Fix $B\in\mathfrak M$ and observe that
\begin{equation*}
\sup_{x\in[0,1]}|B_*(x)|\leq\sup_{x\in[0,1]}\|B\|\sup_{k\in\mathbb N}|g_k(x)|\leq\sup_{k\in\mathbb N}\|g_k\|<+\infty.
\end{equation*}
Thus $B_*\in B([0,1],\mathbb R)$.  Since $B\in\mathfrak M$, it follows that $B_*\in\mathcal M([0,1],\mathbb R)$. Moreover, it is easy to check that $B_*(0)=B_*(1)=0$. In consequence, $B_*\in\mathcal M_0^0$. 

Applying properties of medial limits we obtain
\begin{align*}
TB_*(x)&=\int_{\Omega}B\left(\big(g_k(f(x,\omega))\big)_{k\in\mathbb N}\right)dP(\omega)
=B\left(\Big(\int_{\Omega}g_k(f(x,\omega))dP(\omega)\Big)_{k\in\mathbb N}\right)\\
&=B\left(\Big(\sum_{l=0}^{k-1}\int_{\Omega}T^{l}g(f(\omega,x))dP(\omega)\Big)_{k\in\mathbb N}\right)=B\left(\Big(\sum_{l=0}^{k-1}T^{l+1}g(x)\Big)_{k\in\mathbb N}\right)\\
&=B\left(\big(g_{k+1}(x)\big)_{k\in\mathbb N}-(g(x))_{k\in\mathbb N}\right)=B_*(x)-g(x).
\end{align*}
for every $x\in[0,1]$. 
\end{proof}

We now want to find conditions under which $B_*\in\mathcal B_0^0$. The situation is similar to that for $B_h\in\mathcal B_a^b$. Namely, to prove that $B_*\in\mathcal B_0^0$, we would have to show that $TB_*\in\mathcal B_0^0$, but by Lemma \ref{lemB*} we have $TB_*=B_*-g$. This leads us to the following definition. 

We say that a function $g\in\mathcal B_0^0$ is {\it admissible for $B\in\mathfrak B$}, if the family $\mathcal G$ is bounded and $B_*\in\mathcal B_0^0$. 

Note that the assumption on boundedness of $\mathcal G$ in the admissibility definition is not restrictive, because if the family $\mathcal G$ is unbounded, then $B_*$ can not be a solution of equation \eqref{E} by Lemma \ref{lemG}.

Before we give examples of conditions guaranteeing admissibility of a given function under a Banach limit, let us recall the definition of almost convergence of sequences. Namely, a bounded sequence $(x_{m})_{m\in\mathbb N}$ of real numbers is said to be almost convergent to a real number $x$ if $B((x_{m})_{m\in\mathbb N})=x$ for any $B\in\mathfrak B$. The sequence $(0,1,0,1,0,1,\ldots)$ is a simple example of a non-convergent sequence which is almost convergent. However almost none of the sequences consisting of $0$'s and $1$'s are almost convergent (see \cite{C1990}). It is proved in \cite{L1948} that a sequence $(x_{k})_{k\in\mathbb N}$ is almost convergent to $x$ if and only if $\lim_{n\to \infty }\frac {1}{n}\sum_{m=0}^{n-1}x_{k+m}=x$ uniformly in $k$. Therefore, for a given $x\in[0,1]$ there exists $y\in\mathbb R$ such that  $B(x)=y$ for every $B\in\mathfrak B$ if and only if 
\begin{equation*}
\lim_{n\to\infty}\left(\sum_{l=0}^{n+k-2}T^lg(x)-\sum_{l=k}^{n+k-2}\frac{l+1-k}{n}T^lg(x)\right)=y\quad\hbox{uniformly in }k.
\end{equation*}

\begin{example}\label{exA}
Assume that $\mathcal G\subset\mathcal B_0^0$. If the series $\sum_{l=0}^{\infty}T^lg$ pointwise almost converges to a function from $\mathcal B_0^0$, then $g$ is admissible for every $B\in\mathfrak B$. 
\end{example}

Observe that if $\mathcal G\subset\mathcal B_0^0$ and if the series $\sum_{l=0}^{\infty}T^lg$ pointwise converges to a bounded function, then the series pointwise almost converges to the same bounded function and $B_*=\sum_{l=0}^{\infty}T^lg$ for any $B\in\mathfrak B$. Moreover, since
\begin{equation*}
B_*(x)=B((T^mg_{k}(x))_{k\in\mathbb N})+\sum_{l=0}^{m}T^lg(x)
\end{equation*}
for all $x\in[0,1]$, $B\in\mathfrak B$ and $m\in\mathbb N$, it follows that for a fixed $x\in[0,1]$  the series $\sum_{l=0}^{\infty}T^lg(x)$ converges if and only if the limit $\lim_{m\to\infty}B((T^mg_{k}(x))_{k\in\mathbb N})$ exists for every $B\in\mathfrak B$.

\begin{example}\label{exB}
Assume that $\mathcal B([0,1],\mathbb R)=B([0,1],\mathbb R)$. Then every function $g\in\mathcal B_0^0$ guaranteeing boundedness of the family $\mathcal G$ is admissible for any $B\in\mathfrak B$.
\end{example}

\begin{example}\label{exD}
Assume that $g\in\mathcal B_0^0$ and there exists $m\in\mathbb N$ such that
\begin{equation}\label{Tmg=0}
T^{m}g=0.
\end{equation} 
Then $\mathcal G=\big\{\sum_{l=0}^{k-1}T^lg:k\in\{1,\ldots,m\}\big\}$ and $B_*=\sum_{l=0}^{m-1}T^lg$ for any $B\in\mathfrak B$. Therefore $g$ is admissible for any $B\in\mathfrak B$.
\end{example}

Let us note that condition \eqref{Tmg=0} is not very far from a necessary condition for $g$ derived in Lemma \ref{lemBg}, which says that $B\big((T^mg(x))_{m\in\mathbb N}\big)=0$ for all $x\in[0,1]$ and $B\in\mathfrak B$. 

We now formulate the main result of this paper. 

\begin{theorem}\label{thmmainE}
\begin{enumerate}
\item[\rm (i)] Assume that
\begin{equation}\label{notempty}
\mathfrak{sol}_a^b\eqref{E}\neq\emptyset.
\end{equation} 
Then for every $B\in\mathfrak M$ we have $\mathfrak{sol}_a^b\eqref{E}\subset\left\{B_h+B_*:h\in\mathcal B_a^b\right\}$.
Moreover, $g$ is admissible for any $B\in\mathfrak B$ under which $\mathcal B_a^b$ is closed.
\item[\rm (ii)] If $\mathcal B_a^b$ is closed under $B\in\mathfrak M$ and $g\in\mathcal B_0^0$ is admissible for $B$, then $\mathfrak{sol}_a^b\eqref{E}=\left\{B_h+B_*:h\in\mathcal B_a^b\right\}$. 
\item[\rm (iii)]  If $g\in\mathcal B_0^0$ is admissible for $B\in\mathfrak B$, then
$\mathfrak{sol}_a^b\eqref{E}=\mathfrak{sol}_a^b\eqref{E0}+B_*$.
\end{enumerate}	
\end{theorem}

\begin{proof}
(i) Fix $\varphi\in\mathfrak{sol}_a^b\eqref{E}$ and $B\in\mathfrak B$. Obviously, $B_\varphi$ is well defined. From Lemma \ref{lemG} we conclude that $B_*$ is also well defined. Applying induction to \eqref{eT} we get
\begin{equation}\label{Tkgk}
\varphi=T^k\varphi+g_k
\end{equation}
for every $k\in\mathbb N$, and hence $\varphi(x)=B_\varphi(x)+B_*(x)$ for every $x\in[0,1]$. Thus $\mathfrak{sol}_a^b\eqref{E}\subset\left\{B_h+B_*:h\in\mathcal B_a^b\right\}$. Moreover, if $\mathcal B_a^b$ is closed under $B\in\mathfrak B$, then $B_\varphi\in\mathcal B_a^b$, and making use of \eqref{Tkgk} we obtain that $\sup_{k\in\mathbb N}\|g_k\|\leq2\|\varphi\|$ and $B_*=\varphi-B_\varphi\in\mathcal B_0^0$.

To prove that $\left\{B_h+B_*:h\in\mathcal B_a^b\right\}\subset\mathfrak{sol}_a^b\eqref{E}$ we fix $B\in\mathfrak M$, $h\in\mathcal B_a^b$  and assume that $B_h+B_*\in\mathcal B_a^b$. Then Lemmas \ref{lemBh}, \ref{lemG} and \ref{lemB*} give
\begin{equation*}
T(B_h+B_*)+g=TB_h+TB_*+g=B_h+B_*,
\end{equation*}
which means that $B_h+B_*\in \mathfrak{sol}_a^b\eqref{E}$.

(ii) It suffices to apply assertion (i).

(iii) Fix $\varphi\in\mathfrak{sol}_a^b\eqref{E}$. Lemma \ref{lemB*} jointly with the admissibility of $g$ implies that $\varphi-B_*\in\mathfrak{sol}_a^b\eqref{E0}$. Hence $\varphi=(\varphi-B_*)+B_*\in\mathfrak{sol}_a^b\eqref{E0}+B_*$.
Conversely, fix $\Phi\in\mathfrak{sol}_a^b\eqref{E0}$. Then again Lemma \ref{lemB*} jointly with the admissibility of $g$ implies that $\Phi+B_*\in\mathfrak{sol}_a^b\eqref{E}$.
\end{proof}

\begin{corollary}\label{cormainE}
Assume that $g\in\mathcal B_0^0$ and $\mathcal B_a^b$ is closed under $B\in\mathfrak M$. Then equation \eqref{E} has a solution in $\mathcal B_a^b$ if and only if $g$ is admissible for $B$ and equation \eqref{E0} has a solution in $\mathcal B_a^b$.
\end{corollary}

\begin{remark}
If $\mathfrak{sol}_a^b\eqref{E}=\emptyset$, then it may happen that there is no $B\in\mathfrak B$ for which $B_*$ is well defined; see e.g. the equation $\varphi(x)=\varphi(x)+1$. Therefore, assumption \eqref{notempty} can not be omitted in assertion (i) of Theorem \ref{thmmainE}. The above exemplary equation also  shows that the admissibility assumption in assertion (iii) of Theorem \ref{thmmainE} is necessary.
\end{remark}


\section{Consequence of the main results}
In this section we formulate some exemplary consequences of the main results, making use of the presented examples and applying some know results on equation \eqref{E}. We begin with the case where $\mathcal A=2^\Omega$.

\begin{corollary}\label{cor1}
Assume 
\begin{enumerate}
\item[\rm (H$_4$)]  $(f_n)_{n\in\mathbb N}$ is a sequence of self-mappings of $[0,1]$ such that $f_n(0)=0$ and $f_n(1)=1$ for every $n\in\mathbb N$, $(p_n)_{n\in\mathbb N}$ is a sequence of nonnegative real numbers summing up to one and $g\in B([0,1],\mathbb R)_0^0$.
\end{enumerate} 
Then the equation
\begin{equation}\label{e}
\varphi(x)=\sum_{n\in\mathbb N}p_n\varphi(f_n(x))+g(x)\tag{e$_g$}
\end{equation}
has a solution in $B([0,1],\mathbb R)$ if and only if the family
\begin{equation}\label{family}
\left\{\sum_{l=1}^{k}\sum_{n_1,\ldots,n_l=0}^Np_{n_1}\cdots p_{n_l}(g\circ f_{n_1}\circ\dots\circ f_{n_l}):k\in\mathbb N\right\}
\end{equation}
is bounded. Moreover, if the family given by \eqref{family}  is bounded, then $\varphi\in B([0,1],\mathbb R)$ is a solution of equation \eqref{e} if and only if
$\varphi=B_h+B_*$ with some $B\in\mathfrak B$ and $h\in B([0,1],\mathbb R)$.
\end{corollary}

\begin{proof}
In view of Examples \ref{ex1}, \ref{ex1a} and \ref{exB}, it suffices to apply Theorem \ref{thmmainE0} and Corollary \ref{cormainE} with $\mathcal B([0,1],\mathbb R)=B([0,1],\mathbb R)$ and arbitrary $B\in\mathfrak B$.
\end{proof}

Now we show a possible application of Corollary \ref{cor1}.

\begin{example}
Fix  $N\in\mathbb N$, real numbers $p_0,\ldots,p_N\geq 0$ summing up to one and a function $f\colon[0,1]\to[0,1]$ such that $f(0)=0$, $f(1)=1$ and $f^{N+1}(x)=x$ for every $x\in[0,1]$; for a full description of such functions see \cite[Theorem 15.1]{K1968}. Then consider the following functional equation
\begin{eqnarray}\label{eq1}
\varphi(x)=\sum_{n=0}^{N}p_n\varphi(f^n(x))+g(x),
\end{eqnarray}
which is discussed in more details in \cite[Chapter XIII]{K1968} and in \cite[Subsections 6.3 and 6.7]{KCG1990}).

For all $n\in\{0,\ldots,N\}$ and $m\in\mathbb N$ define recursively numbers $\alpha_{m,n}$ putting
\begin{equation*}
\alpha_{1,n}=p_n\quad\hbox{ and }\quad
\alpha_{m+1,n}=\sum_{k=0}^N\alpha_{m,k}p_{(n-k)\hspace*{-1.9ex}\mod\!\!(N+1)}.
\end{equation*}
Fix $h\in B([0,1],\mathbb R)$ and $B\in\mathfrak B$. Applying induction we obtain 
\begin{equation*}\label{Tmh}
T^mh(x)=\sum_{n=0}^{N}\alpha_{m,n}h(f^n(x))
\end{equation*} 
for all $m\in\mathbb N$ and $x\in[0,1]$. Therefore, from Corollary \ref{cor1} we infer that equation \eqref{eq1} has a solution in $B([0,1],\mathbb R)$ if and only if the family
\begin{equation}\label{family1}
\left\{\sum_{n=0}^{N}\Big(\sum_{m=1}^{k}\alpha_{m,n}\Big)g\circ f^n:k\in\mathbb N\right\}
\end{equation}
is bounded. Moreover, if the family \eqref{family1} is bounded, then $\varphi\colon[0,1]\to\mathbb R$ is a bounded solution of equation \eqref{eq1} if and only if
\begin{equation*}
\varphi(x)=\sum_{n=0}^{N}B\left(\big(\alpha_{m,n}\big)_{m\in\mathbb N}\right)h(f^n(x))+g(x)+
B\left(\sum_{n=0}^{N}g(f^n(x))\Big(\sum_{m=1}^{k}\alpha_{m,n}\Big)_{k\in\mathbb N}\right)
\end{equation*}
with some $h\in B([0,1],\mathbb R)$ and $B\in\mathfrak B$. 

If $p_0=\dots=p_N=\frac{1}{N+1}$, then $\alpha_{m,n}=\frac{1}{N+1}$ for all $m\in\mathbb N$ and $n\in\{0,\ldots,N\}$, and hence the family \eqref{family1} is bounded if and only if 
\begin{equation}\label{condeq1}
\sum_{n=0}^{N}g(f^n(x))=0\quad\hbox{ for every }x\in[0,1];
\end{equation} cf. Example \ref{exD}. In consequence, equation \eqref{eq1} with $p_0=\dots=p_N=\frac{1}{N+1}$ has a solution $\varphi\in B([0,1],\mathbb R)$ if and only if \eqref{condeq1} holds, and moreover,
\begin{equation*}
\varphi(x)=\frac{1}{N+1}\sum_{n=0}^{N}h(f^n(x))+g(x)
\end{equation*}
with some $h\in B([0,1],\mathbb R)$.
\end{example} 

Purely bounded solutions of equation \eqref{E} are considered rather rarely. Usually some additional property is requited, such as monotonicity (see e.g. \cite{KM2009,KM2010,MR2008}), Borel measurability (see e.g. \cite{B2009b,BK2009}), continuity at a point (see e.g. \cite{BJ2004}). The next two corollaries concern just such cases. To formulate the first one we need some notion. Namely, following \cite{BK1977} (cf. \cite{D1977}) we define iterates of a function $h\colon[0,1]\times\Omega\to[0,1]$ as follows
\begin{equation*}
h(x,\omega)=h(x,\omega_1)\quad\hbox{ and }\quad
h^{n+1}(x,\omega)=h(h^n(x,\omega),\omega_{n+1})
\end{equation*}
for all $x\in[0,1]$, $\omega=(\omega_1,\omega_2,\ldots)\in\Omega^\infty$ and $n\in\mathbb N$. Note that if $h$ is an rv-function, then all its iterates are also rv-functions defined on the product space $(\Omega^\infty,\mathcal A^\infty,P^\infty)$. 

\begin{corollary}\label{cor3}
Assume that $f$ is an rv-function such that the function $f(\cdot,\omega)$ is continuous at $0$ and $1$ for every $\omega\in\Omega$ and the function $m\colon[0,1]\to[0,1]$ defined by $m(x)=\int_\Omega f(x,\omega)dP(\omega)$ is continuous with $m(x)\neq x$ for every $x\in(0,1)$. Let $g\in B([0,1],\mathbb R)_0^0$ be Borel measurable continuous at $0$ and $1$, let the family $\mathcal G$ be bounded, and let $B$ be a medial limit with respect to a probability Borel measure on $[0,1]$ such that $B_*$ is continuous at $0$ and $1$. If $\varphi\in B([0,1],\mathbb R)_0^1$ is a Borel measurable, continuous at $0$ and $1$ solution of equation \eqref{E}, then
\begin{equation*}
\varphi(x)=P^\infty\left(\lim_{n\to\infty}f^n(x,\cdot)=1\right)+B_*(x)
\end{equation*}
for every $x\in[0,1]$. 
\end{corollary}

\begin{proof}
Choose $\mathcal B([0,1],\mathbb R)=BM([0,1],\mathbb R)\cap C_0([0,1],\mathbb R)\cap C_1([0,1],\mathbb R)$; this is possible in view of Examples \ref{ex7} and \ref{ex2}. According to \cite[Proposition 2.1 and Corollary 2.4]{BJ2004} we have $\mathfrak{sol}_a^b\eqref{E0}=\{\Phi\}$, where $\Phi(x)=P^\infty\left(\lim_{n\to\infty}f^n(x,\cdot)=1\right)$ for every $x\in[0,1]$. Finally, since $g$ is admissible for $B$, it is enough to apply Theorem \ref{thmmainE}(iii).
\end{proof}

\begin{corollary}\label{cor2}
Assume {\rm (H$_4$)}. Let $x_0\in\{0,1\}$ and let there exists $\eta>0$ such that $\frac{f_n(x)-f_n(x_0)}{x-x_0}\leq 1$ for all $n\in\mathbb N$ and $x\in(0,1)$ with $|x-x_0|\leq\eta$. If $g\in C_{x_0}([0,1],\mathbb R)_0^0$ and the series $\sum_{l=0}^{\infty}T^lg$ converges uniformly, then $\varphi\in C_{x_0}([0,1],\mathbb R)$ is a solution of equation \eqref{e} if and only if there exists $h\in C_{x_0}([0,1],\mathbb R)$ such that $\varphi=B_h+\sum_{l=0}^{\infty}T^lg$ with an arbitrary $B\in\mathfrak B$.
\end{corollary}

\begin{proof}
The uniform convergence of the series $\sum_{l=0}^{\infty}T^lg$ implies its pointwise almost convergence to a function from the class $C_{x_0}([0,1],\mathbb R)$ as well as the boundedness of the family $\mathcal G$. Now it is enough to apply Theorems \ref{thmmainE0} and \ref{thmmainE} with $\mathcal B([0,1],\mathbb R)=C_{x_0}([0,1],\mathbb R)$ and an arbitrary $B\in\mathfrak B$, which is possible in view of Examples \ref{ex2}, \ref{ex2a} and \ref{exA}.
\end{proof}

The next example is in the spirit of the idea of the manuscript \cite{S} with the use of Corollary \ref{cor2}.

\begin{example}
Assume (H$_4$) with $g\in C_0([0,1],\mathbb R)$ and let there exists $\alpha>1$ such that $f_n(x)\leq x^\alpha$ for all $n\in\mathbb N$ and $x\in[0.1]$. Then consider equation \ref{e} and its solutions in the class $C_0([0,1],\mathbb R)$.

Fix $h\in C_0([0,1],\mathbb R)$ and $x\in(0,1)$. By induction on $m$ we obtain $T^mh(x)=\sum_{n_1,\ldots,n_m\in\mathbb N}p_{n_1}\cdots p_{n_m}h(f_{n_1}(\ldots( f_{n_m}(x))\ldots))$ and $f_{n_1}(\ldots (f_{n_m}(x))\ldots)\leq x^{\alpha^m}$ for every $m\in\mathbb N$. Thus $\lim_{m\to\infty}T^mh(x)=0$, and hence
\begin{equation*}
B_h(x)=B((T^mh(x))_{m\in\mathbb N})=\begin{cases}h(0),&\hbox{if }x\in[0,1),\\h(1),&\hbox{if }x=1\end{cases}
\end{equation*} 
for every $B\in\mathfrak B$. If the series $\sum_{l=0}^{\infty}T^lg(x)$ uniformly converges, then Corollary \ref{cor2} implies that every solution $\varphi\in C_0([0,1],\mathbb R)$ of equation \eqref{e} is of the form 
\begin{equation*}
\varphi(x)=\begin{cases}
a+\sum_{l=0}^{\infty}T^lg(x),&\hbox{if }x\in[0,1),\\b,&\hbox{if }x=1,\end{cases}
\end{equation*}
where $a,b\in\mathbb R$.
\end{example}

Lipschitzian solutions of equation \eqref{E}, in a more general setting than in this paper, were recently examined in \cite{B2009a,BKM2016,BM2016,BM2017}. However, the next Corollary gives a general formulae for a wide class of Lipschitzian solutions of equation \eqref{E}, in contrast to the papers mentioned, in which assumptions made force uniqueness or uniqueness up to an additive constant of Lipschitzian solutions of the equation considered. 

\begin{corollary}\label{cor4}
Assume $({\rm H}_2)$ and let $g\in Lip([0,1],\mathbb R)$. Then equation \eqref{E} has a solution in $Lip([0,1],\mathbb R)$ if and only if $g$ is admissible for $B\in\mathfrak M$. Moreover, every solution $\varphi\in Lip([0,1],\mathbb R)$ of equation \eqref{E} is of the form $\varphi=B_h+B_*$ with some $h\in Lip([0,1],\mathbb R)$ and $B\in\mathfrak M$.
\end{corollary}

\begin{proof}
First note that (H$_2$) jointly with \eqref{cond} yields
\begin{equation*}\label{fxx}
\int_{\Omega}f(x,\omega)dP(\omega)=x\quad\hbox{ for every }x\in [0,1].
\end{equation*}
This condition implies that each piecewise affine function is a solution of equation \eqref{E0}. In particular, equation \eqref{E0} has a Lipschitzian solution. Now, in view of Examples \ref{ex4} and \ref{ex4a}, it suffice to apply Corollary \ref{cormainE} and Theorem \ref{thmmainE}(ii) with $\mathcal B([0,1],\mathbb R)=Lip([0,1],\mathbb R)$ and suitable $B\in\mathfrak M$.
\end{proof}

The next corollary gives a formulae for the general solution of equation \eqref{E} in the space $BV([0,1],\mathbb R)$, and hence, partially solves the problem considered in \cite{M1985} for a very spacial case of equation \eqref{E0}. 

\begin{corollary}\label{cor5}
Assume $({\rm H}_1)$. Let $g\in BV([0,1],\mathbb R)$. Then equation \eqref{E} has a solution in $BV([0,1],\mathbb R)$ if and only if $g$ is admissible for some $B\in\mathfrak M$ and equation \eqref{E0} has a solution $\Phi\in BV([0,1],\mathbb R)$. Moreover, $\varphi\in BV([0,1],\mathbb R)$ satisfies \eqref{E} if and only if $\varphi=B_h+B_*$ with some $h\in BV([0,1],\mathbb R)$ and $B\in\mathfrak M$.
\end{corollary}

\begin{proof}
It is enough to apply Theorems \ref{thmmainE0} and \ref{thmmainE} with $\mathcal B([0,1],\mathbb R)=BV([0,1],\mathbb R)$ and arbitrary $B\in\mathfrak M$, which is possible in view of Examples \ref{ex5} and \ref{ex5a}.
\end{proof}

Before we formulate the last corollary of this paper let us to extend the main result of \cite{M1985} to equation \eqref{E0}.

\begin{proposition}
Assume $({\rm H}_1)$. If $\Phi\in BV([0,1],\mathbb R)$ satisfies \eqref{E0}, then also $\Phi_+$ and $\Phi_-$ satisfy \eqref{E0}.
\end{proposition}

\begin{proof}
Fix $\Phi\in BV([0,1],\mathbb R)$ satisfying \eqref{E0}. Define functions $F,G\colon[0,1]\to\mathbb R$ by putting $F(x)=\int_{\Omega}\Phi_+(f(x,\omega))dP(\omega)$ and $G(x)=\int_{\Omega}\Phi_-(f(x,\omega))dP(\omega)$, where $\Phi_+$ and $\Phi_-$ are the upper and the lower variation (obtained by the Jordan decomposition) of $\Phi$, respectively. Then $\Phi_+-\Phi_-=F-G$ and by (H$_1$) both the functions $F$ and $G$ are increasing. Hence the Jordan decomposition yields $\Phi_+(y)-\Phi_+(x)\leq F(y)-F(x)$ for all $0\leq x\leq y\leq 1$. Putting in the last inequality $x=0$ and $y=1$ in turn and making use of \eqref{cond} we obtain $\Phi_+(y)\leq F(y)$ and $\Phi_+(x)\geq F(x)$ for all $x,y\in[0,1]$. In consequence, $\Phi_+=F$ and $\Phi_-=G$.
\end{proof}

The above proposition reduces the problem of determining all solutions of bounded variation of equation \eqref{E0} to that of finding all increasing solutions of this equation. We end this paper determining all increasing solutions of equation \eqref{E0}.

\begin{corollary}\label{cor6}
Assume $({\rm H}_1)$. Then $\varphi\colon [0,1]\to\mathbb R$ is an increasing solution of equation \eqref{E0} if and only if $\varphi=B_h$ with some increasing function $h\colon [0,1]\to\mathbb R$ and $B\in\mathfrak M$.
\end{corollary}

\begin{proof}
If $\varphi\colon [0,1]\to\mathbb R$ is an increasing solution of equation \eqref{E0}, then $\varphi=B_\varphi$ with any $B\in\mathfrak M$.

Conversely, if $h\colon [0,1]\to\mathbb R$ is increasing and $B\in\mathfrak M$, then $B_h$ is increasing as well. Moreover, Corollary \ref{cor5} implies that $B_h$ satisfies \eqref{E0}.
\end{proof}

\section*{Acknowledgments}

This research was supported by the University of Silesia Mathematics Department (Iterative Functional Equations and Real Analysis program)

\bibliographystyle{plain}
\bibliography{bibliography}

\begin{thebibliography}{10}

\bibitem{B1955}
Stefan Banach.
\newblock {\em Th\'{e}orie des op\'{e}rations lin\'{e}aires}.
\newblock Chelsea Publishing Co., New York, 1955.

\bibitem{B2009b}
Karol Baron.
\newblock Linear iterative equations of higher orders and random-valued
  functions.
\newblock {\em Publ. Math. Debrecen}, 75(1-2):1--9, 2009.

\bibitem{B2009a}
Karol Baron.
\newblock On the convergence in law of iterates of random-valued functions.
\newblock {\em Aust. J. Math. Anal. Appl.}, 6(1):Art. 3, 9, 2009.

\bibitem{BJ2001}
Karol Baron and Witold Jarczyk.
\newblock Recent results on functional equations in a single variable,
  perspectives and open problems.
\newblock {\em Aequationes Math.}, 61(1-2):1--48, 2001.

\bibitem{BJ2004}
Karol Baron and Witold Jarczyk.
\newblock Random-valued functions and iterative functional equations.
\newblock {\em Aequationes Math.}, 67(1-2):140--153, 2004.

\bibitem{BK2009}
Karol Baron and Rafa\l Kapica.
\newblock A uniqueness-type problem for linear iterative equations.
\newblock {\em Analysis (Munich)}, 29(1):95--101, 2009.

\bibitem{BKM2016}
Karol Baron, Rafa\l Kapica, and Janusz Morawiec.
\newblock On {L}ipschitzian solutions to an inhomogeneous linear iterative
  equation.
\newblock {\em Aequationes Math.}, 90(1):77--85, 2016.

\bibitem{BK1977}
Karol Baron and Marek Kuczma.
\newblock Iteration of random-valued functions on the unit interval.
\newblock {\em Colloq. Math.}, 37(2):263--269, 1977.

\bibitem{BM2016}
Karol Baron and Janusz Morawiec.
\newblock Lipschitzian solutions to linear iterative equations.
\newblock {\em Publ. Math. Debrecen}, 89(3):277--285, 2016.

\bibitem{BM2017}
Karol Baron and Janusz Morawiec.
\newblock Lipschitzian solutions to linear iterative equations revisited.
\newblock {\em Aequationes Math.}, 91(1):161--167, 2017.

\bibitem{C1969}
Ching Chou.
\newblock On the size of the set of left invariant means on a semi-group.
\newblock {\em Proc. Amer. Math. Soc.}, 23:199--205, 1969.

\bibitem{C1990}
Jeff Connor.
\newblock Almost none of the sequences of {$0$}'s and {$1$}'s are almost
  convergent.
\newblock {\em Internat. J. Math. Math. Sci.}, 13(4):775--777, 1990.

\bibitem{D1977}
Phil Diamond.
\newblock A stochastic functional equation.
\newblock {\em Aequationes Math.}, 15(2-3):225--233, 1977.

\bibitem{D1965}
R.~G. Douglas.
\newblock On the measure-theoretic character of an invariant mean.
\newblock {\em Proc. Amer. Math. Soc.}, 16:30--36, 1965.

\bibitem{F2015}
D.~H. Fremlin.
\newblock {\em Measure theory. {V}ol. 5. {S}et-theoretic measure theory. {P}art
  {I}}.
\newblock Torres Fremlin, Colchester, 2015.
\newblock Corrected reprint of the 2008 original.

\bibitem{KM2009}
Tomasz Kochanek and Janusz Morawiec.
\newblock Probability distribution solutions of a general linear equation of
  infinite order.
\newblock {\em Ann. Polon. Math.}, 95(2):103--114, 2009.

\bibitem{KM2010}
Tomasz Kochanek and Janusz Morawiec.
\newblock Probability distribution solutions of a general linear equation of
  infinite order, {II}.
\newblock {\em Ann. Polon. Math.}, 99(3):215--224, 2010.

\bibitem{K1968}
Marek Kuczma.
\newblock {\em Functional equations in a single variable}.
\newblock Monografie Matematyczne, Tom 46. Pa\'{n}stwowe Wydawnictwo Naukowe,
  Warsaw, 1968.

\bibitem{KCG1990}
Marek Kuczma, Bogdan Choczewski, and Roman Ger.
\newblock {\em Iterative functional equations}, volume~32 of {\em Encyclopedia
  of Mathematics and its Applications}.
\newblock Cambridge University Press, Cambridge, 1990.

\bibitem{L2009}
Paul~B. Larson.
\newblock The filter dichotomy and medial limits.
\newblock {\em J. Math. Log.}, 9(2):159--165, 2009.

\bibitem{L1988}
Stanis{\l}aw {\L}ojasiewicz.
\newblock {\em An introduction to the theory of real functions}.
\newblock A Wiley-Interscience Publication. John Wiley \& Sons, Ltd.,
  Chichester, third edition, 1988.
\newblock With contributions by M. Kosiek, W. Mlak and Z. Opial, Translated
  from the Polish by G. H. Lawden, Translation edited by A. V. Ferreira.

\bibitem{L1948}
G.~G. Lorentz.
\newblock A contribution to the theory of divergent sequences.
\newblock {\em Acta Math.}, 80:167--190, 1948.

\bibitem{M1985}
Janusz Matkowski.
\newblock Remark on {BV}-solutions of a functional equation connected with
  invariant measures.
\newblock {\em Aequationes Math.}, 29(2-3):210--213, 1985.

\bibitem{M1973}
P.~A. Meyer.
\newblock Limites m\'{e}diales, d'apr\`es {M}okobodzki.
\newblock pages 198--204. Lecture Notes in Math., Vol. 321, 1973.

\bibitem{M1969}
Gabriel Mokobodzki.
\newblock Ultrafiltres rapides sur {N}. {C}onstruction d'une densit\'{e}
  relative de deux potentiels comparables.
\newblock In {\em S\'{e}minaire de {T}h\'{e}orie du {P}otentiel, dirig\'{e} par
  {M}. {B}relot, {G}. {C}hoquet et {J}. {D}eny: 1967/68, {E}xp. 12}, page~22.
  Secr\'{e}tariat math\'{e}matique, Paris, 1969.

\bibitem{M1995}
Janusz Morawiec.
\newblock On a linear functional equation.
\newblock {\em Bull. Polish Acad. Sci. Math.}, 43(2):131--142, 1995.

\bibitem{M1998}
Janusz Morawiec.
\newblock Some properties of probability distribution solutions of linear
  functional equations.
\newblock {\em Aequationes Math.}, 56(1-2):81--90, 1998.

\bibitem{MR2008}
Janusz Morawiec and Ludwig Reich.
\newblock The set of probability distribution solutions of a linear functional
  equation.
\newblock {\em Ann. Polon. Math.}, 93(3):253--261, 2008.

\bibitem{P2014}
Albrecht Pietsch.
\newblock Traces of operators and their history.
\newblock {\em Acta Comment. Univ. Tartu. Math.}, 18(1):51--64, 2014.

\bibitem{S}
Mariusz Sudzik.
\newblock Iterative functional equations and global attractive fixed points.
\newblock manuscript.

\bibitem{W2000}
Benjamin Weiss.
\newblock A survey of generic dynamics.
\newblock In {\em Descriptive set theory and dynamical systems
  ({M}arseille-{L}uminy, 1996)}, volume 277 of {\em London Math. Soc. Lecture
  Note Ser.}, pages 273--291. Cambridge Univ. Press, Cambridge, 2000.

\end{thebibliography}

\end{document}